\documentclass{article}
\usepackage{amsfonts, amsmath,amscd, amssymb, latexsym, mathrsfs, verbatim, wasysym }
\usepackage{graphicx}
\usepackage{hyperref}
\usepackage{bm}
\usepackage{amsbsy}
\usepackage{color}
\usepackage{epsfig}
\usepackage{psfrag}
\usepackage{subfigure}
\usepackage{captcont}
\usepackage{appendix}
\usepackage{amsthm}
\usepackage{mathtools}

\numberwithin{equation}{section}
\newtheorem{Equation}{}[section]

\newtheorem{theorem}[Equation]{Theorem}

\newtheorem{lemma}[Equation]{Lemma}
\newtheorem{corollary}[Equation]{Corollary}

\newtheorem{definition}[Equation]{Definition}

\newtheorem{remark}[Equation]{Remark}

\usepackage{color}
\definecolor{darkgreen}{cmyk}{1,0,1,.2}
\definecolor{m}{rgb}{1,0.1,1}
\definecolor{green}{cmyk}{1,0,1,0}
\definecolor{darkred}{rgb}{0.55, 0.0, 0.0}
\definecolor{test}{rgb}{1,0,0}
\definecolor{cmyk}{cmyk}{0,1,1,0}

\def\R{\mathbb R}

\newcommand{\floor}[1]{\lfloor #1 \rfloor}
\usepackage[top=1in, bottom=1.25in, left=1.25in, right=1.25in]{geometry}
\begin{document}
\markboth{Mahashweta Patra et al.}{Shilnikov-type Dynamics in Three-Dimensional Piecewise Smooth Maps}
\title{Shilnikov-type Dynamics in Three-Dimensional Piecewise Smooth Maps}
\author{Indrava Roy\\Institute of Mathematical Sciences, {HBNI}, Chennai, India\\\textit{Email address}:\texttt{indrava@imsc.res.in},\and
Mahashweta Patra \\{Indian Institute of Science Education and Research, Kolkata, India}\\\textit{Email address}: \texttt{mp12ip005@iiserkol.ac.in},  \and Soumitro Banerjee \\{Indian Institute of Science Education and Research, Kolkata, India} \\\textit{Email address}:\texttt{soumitro@iiserkol.ac.in}
}

\maketitle

\begin{abstract}
We show the existence of Shilnikov-type dynamics and bifurcation behaviour in general discrete three-dimensional piecewise smooth maps and give analytical results
for the occurence of such dynamical behaviour. Our main example in fact shows a `two-sided' Shilnikov dynamics, i.e. simultaneous looping and homoclinic intersection
of the one-dimensional eigenmanfolds of fixed points on both sides of the border. We also present two complementary methods to analyse the return time of an orbit to the border:
one based on recursion and another based on complex interpolation.
\end{abstract}

\section{Introduction}
The phenomenon of chaos is observed in many nonlinear deterministic systems in
both experimental and computer-simulation contexts. In 1965, L. Shilnikov (also written {\v{S}}il'nikov)
\cite{Shilnikov1} showed that in a continuous-time dynamical system if the real
eigenvalue has a larger magnitude than the real part of the complex conjugate
eigenvalues, then there are horseshoes present in the return maps defined near the
homoclinic orbit. Shilnikov chaos normally appears when a parameter is
varied towards the homoclinic condition associated with a saddle focus. The
criteria is called Shilnikov criteria and orbit is called Shilnikov chaos.

 Shilnikov attractors can be found in many different systems, including
Ro\"ssler system, Arenodo-Coullet-Tresser systems, and Rosenzweig-Mac-Arthur
system. Shilnikov chaos is observed in many physical systems, including a
single mode lasers with feedback (\cite{arecchi1987laser}) and with a saturable
absorber (\cite{de1989instabilities}), the Belousov-Zhabotinskii reaction, a
glow discharge plasma (\cite{braun1992evidence}), an optically bistable device
(\cite{pisarchik2000discrete}), a multimode laser (\cite{viktorov1995shil}) and
some other systems. Theoretical and experimental study of discrete behavior of
Shilnikov chaos is shown in a $CO_2$ laser in
\cite{pisarchik2001theoretical}.

 Attractors of a spiral type that appear in accordance with this scenario in
discrete dynamical systems are called \textit{discrete Shilnikov attractors}. First examples of such attractors were found in 3-D generalized
H\'{e}non map. \cite{pisarchik2000discrete} have reported on the first experimental
observation of the discrete behavior of Shilnikov chaos.

 \cite{zhou2004constructing} have shown that there only exist three kinds of
chaos- homoclinic chaos, heteroclinic chaos, and a combination of homoclinic and
heteroclinic chaos. They construct a new chaotic system of quadratic polynomial
ordinary differential equations (ODE) in three dimensions, which has a single
equilibrium point. They rigorously prove that this system satisfies all
conditions stated in the Shilnikov theorem (\cite{tresser1984some}), which
clearly reveals its chaos formation mechanism and implies the existence of Smale
horseshoes.

 \cite{silva1993shil} has given a brief introduction of Shilnikov's method
to detect analytically the presence of chaos in continuous autonomous systems.
As an application to a piecewise linear system they took Chua's circuit and have
shown homoclinic orbit from a saddle focus as well as heteroclinic orbit between
two saddle foci.

The mechanism of formation of Shilnikov chaos has not yet been investigated in a
piecewise smooth discrete dynamical system. The questions we address in this
paper are- what are the criteria for a Shilnikov chaos to occur in a
piecewise smooth discrete-time dynamical system? What sort of theoretical conditions can be given that can guarantee existence or non-existence of
transverse homoclinic intersections, which will help create more efficient computer programs that can search for such intersections?


Since fixed-points of saddle-focus type do not appear in a 2-D piecewise linear map, to investigate the Shilnikov phenomenon in a piecewise smooth system we take a 3-D piecewise linear normal
form map (\cite{indrava, 3dijbc, patra2017bifurcation}) and answer the questions mentioned above. We derive an analytical condition for the occurrence of a
homoclinic intersection, thereby, the occurrence of a chaotic orbit. This condition is obtained keeping in view that checking the existence of a homoclinic intersection requires
a computer simulation; the implementation of the algorithm is rendered more convenient by our analytic condition.
In particular, it provides a finite range of iterations which should be checked, outside of which no homoclinic intersection is possible. We also show
a motivating numerical example which exhibits Shilnikov-type behaviour, and show the utility of our methods with the example. The two-sided Shilnikov dynamics, (i.e. Shilnikov-type
behaviour exhibited simultaneously by fixed points on both sides of the border) that this example shows is
an interesting feature that occurs in three-dimensional piecewise smooth systems and may well be unique to them.

The paper is divided as follows. The description of the normal form for a three-dimensional piecewise smooth system is given in
Section (\ref{map_description}). A motivating example is shown for a two-sided Shilnikov-type dynamics in Section (\ref{example}). Section (\ref{theory}) is the
main theoretical part of the paper which gives two complementary methods to analyse the existence of transverse homoclinic intersections. The last section contains discussions of the
results obtained and concluding remarks.

{\em{ Acknowledgements.}}
IR thanks A. Prasad for reading a preliminary draft of the paper,
and the Indian Science and Engineering Research Board (SERB) for support via MATRICS project MTR/2017/000835.

\section{3-Dimensional piecewise smooth maps: system description}\label{map_description}
The piecewise linear approximation of a general piecewise smooth 3D system evaluated in a close neighborhood of the border, called the `normal form' map
 \cite{mario-iscas03, indrava, 3dijbc, patra2017bifurcation} is given by
\begin{equation}
X_{n+1}=F_{\mu}(X_n)=
\begin{cases}
A_lX_n+\mu C, &\text{if $x_n\leq 0$}\\
A_rX_n+\mu C, &\text{if $x_n\geq 0$}
\end{cases}
\label{normalform}
\end{equation}
where $X_n=(x_n, y_n, z_n)^\intercal\in \mathbb{R}^3$, $C = (1, 0, 0)^\intercal \in \mathbb{R}^3$ and $\mu$ is a real-valued parameter.
The phase space of this map is divided by the `border' $X_b: \{x=0\}$
into two regions $\mathcal{L}:= \{(x, y, z) \in \mathbb{R}^3 : x\leq
  0\}$ and $\mathcal{R} := \{(x, y, z) \in \mathbb{R}^3 : x \geq 0\}$. We shall frequently refer to $\mathcal{L}$ and $\mathcal{R}$ as the `left' side and the `right' side of the border, respectively.
  In
each region, the dynamics is governed by an affine map and the
equations are continuous across $X_b$.
$A_l$ and $A_r$ are real valued $3\times 3$ matrices
\[ A_l=\left( \begin{array}{ccc}
\tau_l & 1 & 0 \\
-\sigma_l & 0 & 1 \\
\delta_l & 0 & 0 \end{array} \right)
\;\;\; \mbox{and} \;\;\; A_r=\left( \begin{array}{ccc}
\tau_r & 1 & 0 \\
-\sigma_r & 0 & 1 \\
\delta_r & 0 & 0 \end{array} \right)
\]

If $\lambda_1$, $\lambda_2$ and
$\lambda_3$ are the eigenvalues of the Jacobian matrix of the original PWS map
evaluated at a fixed point placed on the left side close to the border,
 then the
parameters of the matrix $A_l$ are simply the trace $\tau_l=\lambda_1
+ \lambda_2 + \lambda_3$, the second trace
$\sigma_l=\lambda_1\lambda_2 + \lambda_2 \lambda_3 + \lambda_3
\lambda_1$ and the determinant $\delta_l= \lambda_1
\lambda_2\lambda_3$. The parameters of the matrix $A_r$ depends, in a
similar manner, on the eigenvalues of the Jacobian matrix computed at
a fixed point located on the right side.

The fixed points of the system in both sides of the boundary are given by

\begin{eqnarray}\nonumber
L^{*}&=&(\frac{\mu}{1-\tau_l+\sigma_l-\delta_l}, \frac{\mu(-\sigma_l+\delta_l)}{1-\tau_l+\sigma_l-\delta_l}, \frac{\mu\delta_l}{1-\tau_l+\sigma_l-\delta_l})\\ \nonumber
R^{*}&=&(\frac{\mu}{1-\tau_r+\sigma_r-\delta_r}, \frac{\mu(-\sigma_r+\delta_r)}{1-\tau_r+\sigma_r-\delta_r}, \frac{\mu\delta_r}{1-\tau_r+\sigma_r-\delta_r})\\ \nonumber
\end{eqnarray}

If $L^*$ lies on the left side of the border, it is called \emph{admissible}, otherwise it is called \emph{virtual}. Similarly, $R^*$ is admissible if it lies on the right side of the border 
and virtual otherwise. We will assume the generic condition $\mu\neq 0$ for our discussions. 

\section{Motivating example for Shilnikov-type dynamics and bifurcation scenario} \label{example}

Let us consider the system given by Equation (\ref{normalform}) for the following parameter values:

$$
\tau_l = 1.0, \quad  \sigma_l = -0.25, \quad \delta_l = 0.3,\quad
\tau_r = 0.58,\quad \sigma_r = 0.38, \quad \delta_r = -1.27, \quad \mu=1.0
$$

For these values, the left fixed point is admissible, located at the point $L^*=(−1.8182, −1, −0.5455)$. The right fixed point is also admissible and located at $(0.4878, −0.8049, −0.6195)$.

The matrix $A_l$ has a positive unstable eigenvalue $\lambda_{1, l}= 1.3499$, and a pair of stable complex eigenvalues $\lambda_{2,l}, \overline{\lambda}_{2,l}=−0.1749 \pm 0.4378i$ with absolute
value $|\lambda_{2,l}| = 0.4714$. Therefore the left fixed point of saddle-focus type. The matrix $A_r$ governing the dynamics on the right side of the border $x=0$,
has a negative stable eigenvalue $\lambda_{1,r}= -0.8211$ and pair of unstable complex eigenvalues $\lambda_{2,r}, \overline{\lambda}_{2,r}= 0.7105 \pm 1.0207i$ with absolute value $1.2437$. Therefore,
$R^*$ is a flip saddle.

At this setting of parameter values, there exists a transverse homooclinic intersection between the 1-dimensional stable manifold $S_r$ and the 2-dimensional unstable plane $U_r$ of $R^*$ 
(see Figure (\ref{Homoclinicintersection} (d)), resulting in
a chaotic attractor which is stable under small perturbations for $\tau_r< 0.62$, see Figure (\ref{Homoclinicintersection} (a)). The Lyapunov exponents and bifurcation diagrams are shown in Figure
(\ref{LyapunovBifurcation}).

\begin{figure}
\centering
    \begin{subfigure}[]
       \centering
       \includegraphics[scale=.45]{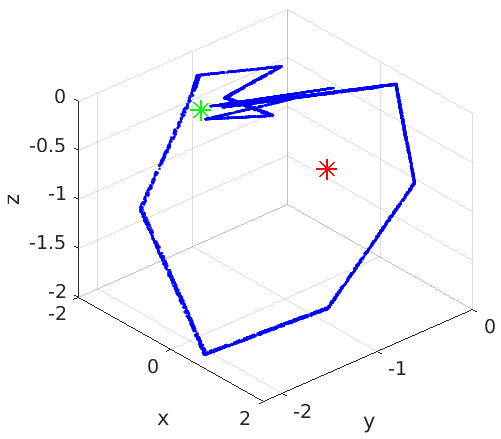}
       \label{attractor}
    \end{subfigure}
    \begin{subfigure}[]
       \centering
       \includegraphics[scale=.45]{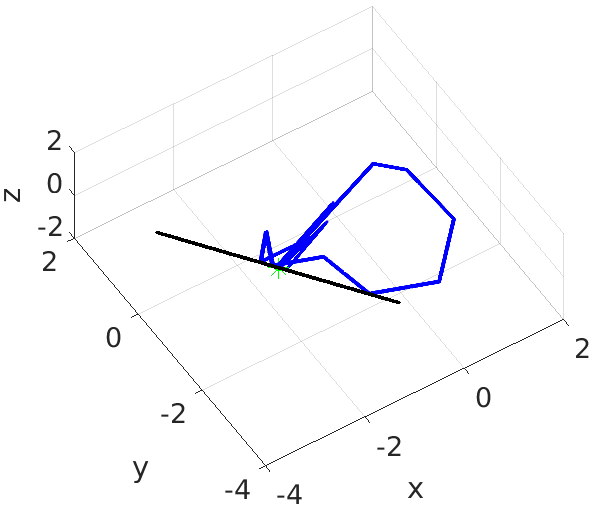}
    \end{subfigure}
    \begin{subfigure}[]
      \centering
      \includegraphics[scale=.45]{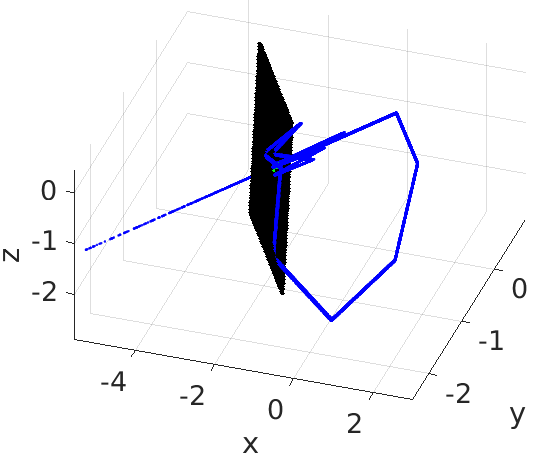}
    \end{subfigure}
     \begin{subfigure}[]
      \centering
      \includegraphics[scale=.45]{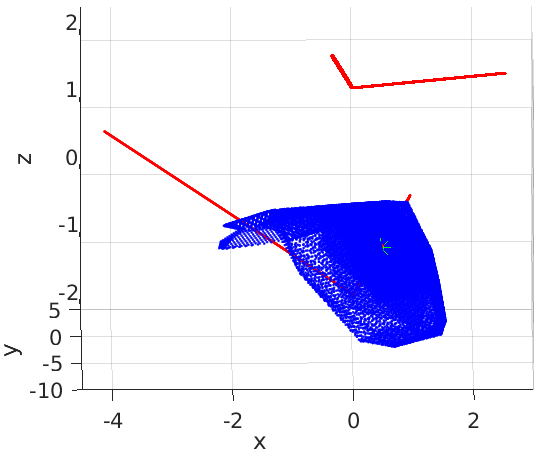}
      \label{UR_SR}
    \end{subfigure}
\caption{(a) Chaotic attractor; (b), (c) Shilnikov-type dynamics and transverse homoclinic intersection between $S_l$(black) and $U_l$(blue) at (b) $\tau_r=0.58$ and (c) $\tau_r =0.62$.
(d) The homoclinic intersection of $U_r$(blue) and $S_r$(red). The rest of the parameters are
the same as in the beginning of Section (\ref{example}); the green and red stars show the location of $L^*$ and $R^*$, respectively.}
\label{Homoclinicintersection}
\end{figure}

We also observe that the unstable manifold $U_l$ for the left fixed point $L^*$ grows until hits the border (say at $B^*_l$) and passes to the right side, and then loops back towards the left side due to the unstable
nature of $R^*$. When $U_l$  crosses the border again, the first iterate that crosses the border lies ``above'' the stable manifold $S_l$, i.e. in the same side of $B^*_l$ with respect to 
$S_l$.
Thus the dynamics mimics the classical Shilnikov looping behaviour of the unstable manifold $U_l$, and a Shilnikov-type bifurcation
occurs for these parameter values when $\tau_r$ increases from $0.58$ to $0.62$. At the critical value $\tau_r=0.62$, the unstable manifold $U_l$ has a transverse homoclinic
intersection with the unstable plane $S_l$, which results in the sudden vanishing of the chaotic attractor (see Figure (\ref{Homoclinicintersection}(c))). Note that due to the transverse intersection of $S_r$ and $U_r$,
chaotic dynamics is still present in the system, but there is no attractor. The basin of attraction is plotted in the planes $z= const.$ in Figure (\ref{Basinsofattraction}), for the values $z=-1.1, 0$ and $1.0$ at
$\tau_r=0.58$ and $\tau_r=0.63$.

\begin{figure}
\centering

    \begin{subfigure}[]
       \centering
       \includegraphics[scale=.35]{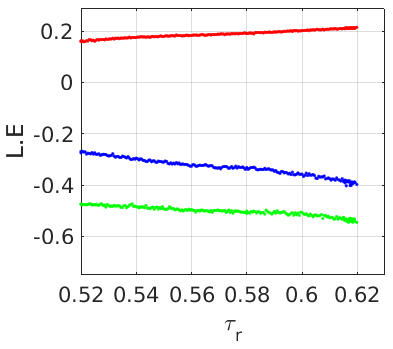}
    \end{subfigure}
    \begin{subfigure}[]
      \centering
      \includegraphics[scale=.35]{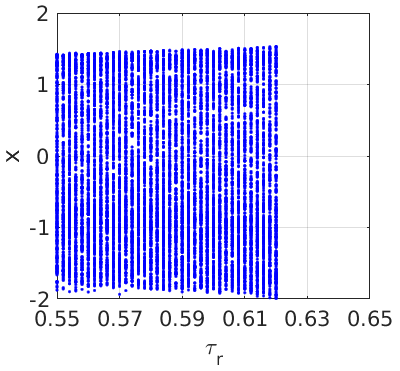}
    \end{subfigure}
\caption{(a). Lyapunov exponent plot, and (b) Bifurcation diagram for the same parameter values as in the beginning of Section (\ref{example}).}
\label{LyapunovBifurcation}
\end{figure}

On the other hand, increasing $\tau_L$ from 1 to $1.05$ causes the homoclinic intersection of $U_r$ and $S_r$ to be broken, and thus the chaotic dynamics also ceases to occur. Thus near the critical
values $\tau_r =0.62$ and $\tau_L=1.05$, a two-sided Shilnikov-type dynamics and bifurcation occurs. This kind of dynamics is most likely unique to piecewise smooth systems due to its inherent
asymmetry.

\begin{figure}
\centering
    \begin{subfigure}[]
       \centering
       \includegraphics[scale=.3]{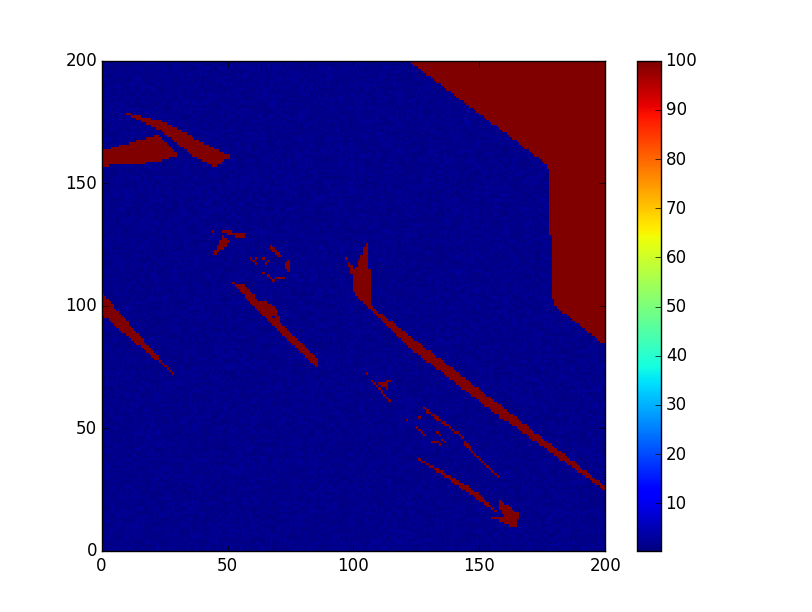}
    \end{subfigure}
    \begin{subfigure}[]
      \centering
      \includegraphics[scale=.3]{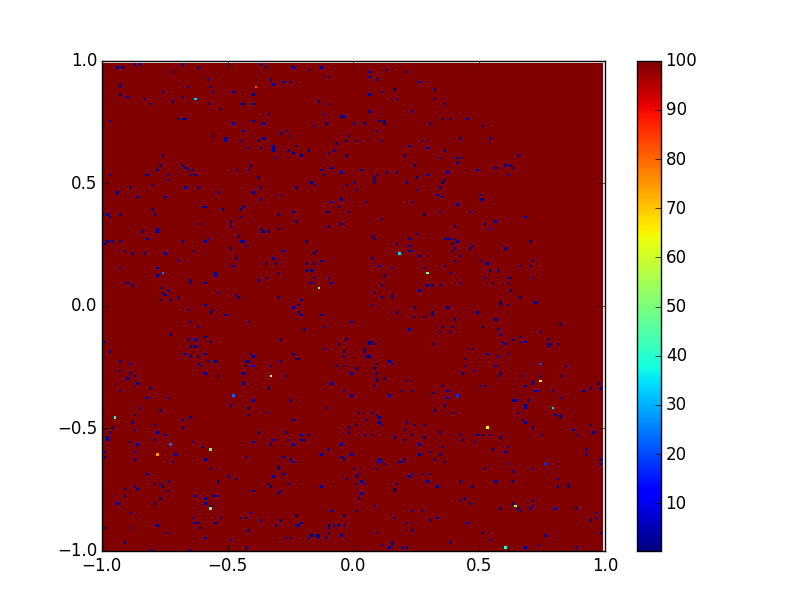}
    \end{subfigure}
    \begin{subfigure}[]
       \centering
       \includegraphics[scale=.3]{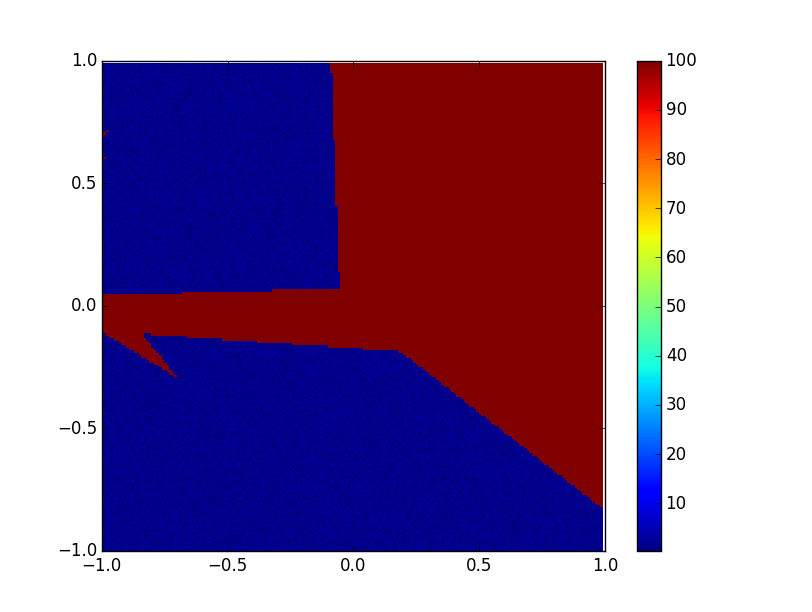}
    \end{subfigure}
    \begin{subfigure}[]
      \centering
      \includegraphics[scale=.3]{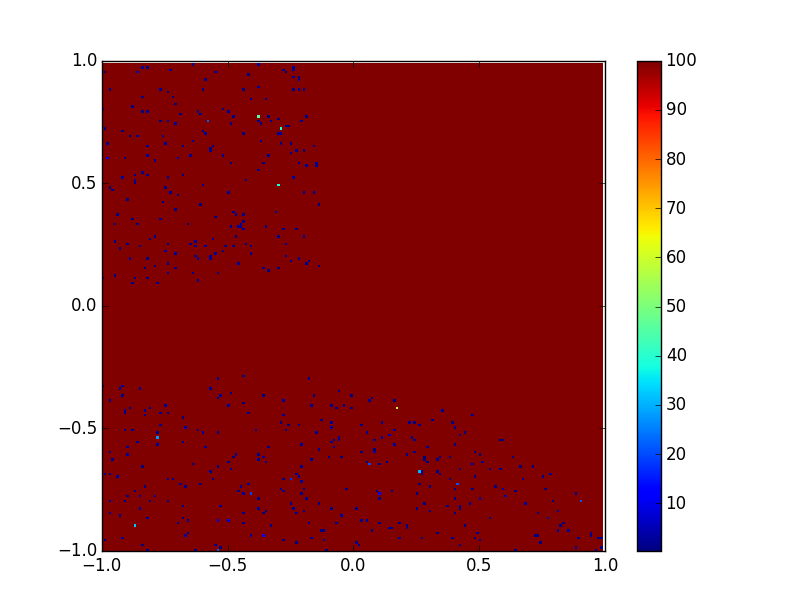}
    \end{subfigure}
    \begin{subfigure}[]
       \centering
       \includegraphics[scale=.3]{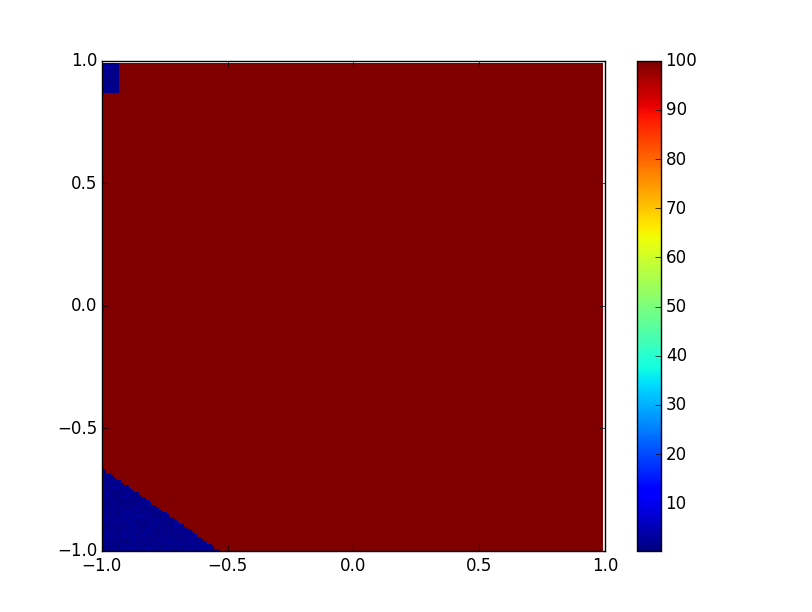}
    \end{subfigure}
    \begin{subfigure}[]
      \centering
      \includegraphics[scale=.3]{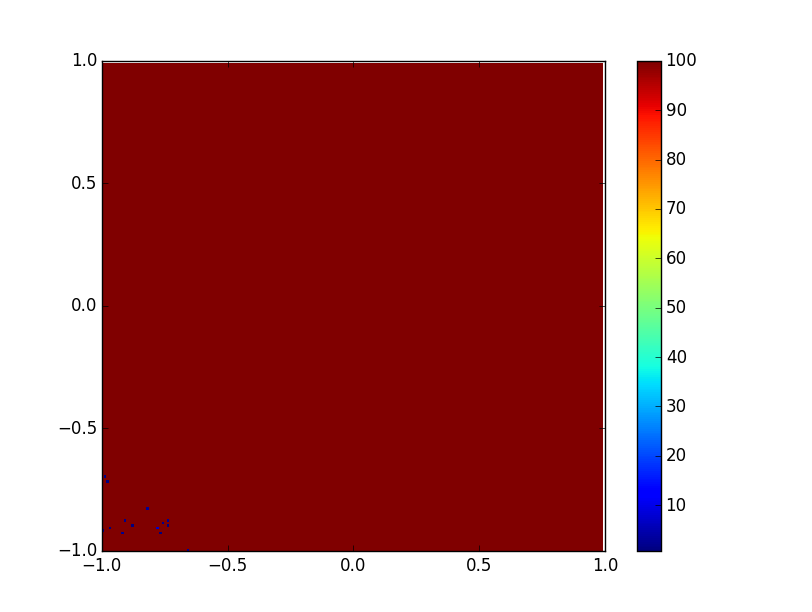}
    \end{subfigure}
\caption{(a), (c), (e). Basin of attraction for the planes $z=-1.1, 0, 1.1$ and $\tau_r=0.58$, (b), (d), (f) Basin of attraction for the planes $z=-1.1, 0, 1.1$ and $\tau_r=0.63$.
The $x$-coordinate is plotted in the horizontal direction and the $y$-coordinate in the vertical direction, both lying in the interval $[-1,1]$.
color bar shows the vector norm of the orbit point after 1000 iterations, with the ceiling at the value 100. Other parameter values are same as in the beginning of Section (\ref{example}).}
\label{Basinsofattraction}
\end{figure}

\section{Analytical conditions for homoclinic intersection} \label{theory}
A homoclinic intersection between stable and unstable manifold implies an infinite number of intersections, therefore a horseshoe structure is born.
We can give a general condition for occurrence chaos through the occurrence of homoclinic intersection.

Let $L^*$ be an admissible fixed point which has an unstable eigenvalue and a pair of stable eigenvalues. The procedure for finding the homoclinic intersection is as follows:

\begin{enumerate}

\item Calculate the unstable eigenvector, calculate the point $P_0$ where it touches the border. Consider that point as the initial point.

\item We calculate the $n$-th iteration of this initial point i.e., $P_n$  where $n$ is the minimum number to cross the border again.

\item Calculate $P_{n+1}$.

\item As the equation of the stable eigenplane is known we can check whether $P_{n+1}$ and $P_n$ are on the same side of the stable eigenplane equation or not.

\item Calculate the intersection between the stable eigenplane and the line going through the $P_n$ and $P_{n+1}$. Check whether the intersection point is in the same side of the fixed point or not.

If the intersection point is on the same side of the fixed point then there is a homoclinic intersection, and therefore, chaotic dynamics must occur.

\end{enumerate}

Here, to calculate the $P_n$, i.e., the $n$-th iteration, we have taken two complementary approaches. The first one is based on a recursion method following \cite{avrutin2016dangerous} (see also\\
\cite{saha2015}). The recursion stops as soon as the orbit crosses the border, at which point it is easily checked whether there exists a transverse homoclinic intersection with the 2-dimensional
stable manifold of $L^*$.

The second method is then used to provide upper bounds for the number of iterations that need to be checked for the border crossing and subsequent homoclinic intersection. This method is based
on a complex interpolation scheme whereby we define fractional iterations of our system in Equation (\ref{normalform}). This interpolation is then used to get a transcendental
equation in the positive reals; the least positive solution of this equation then provides the border return time.

\subsection{Recursive method to compute powers of a matrix}\label{rec}
Let's consider a situation where a period-1 fixed point $L^*= (x_l, y_l, z_l)$ is admissible. The unstable eigenvector grows and touches the $x=0$ plane at $P_0=(0,C_1,C_2)$ and the image of this point is the
1st fold point and its coordinate is $P_1=(x_1,y_1,z_1)=(C_1+\mu, C_2, 0)$. where, $C_1=y_l-\frac{v_{21}}{v_{11}}x_l$, and $C_2=z_l-\frac{v_{31}}{v_{11}}x_l$. Here $V_1=(v_{11}, v_{21}, v_{31})$
is the unstable eigenvector associated to the $A_l$. We shall assume that both matrices $A_l$ and $A_r$ are non-singular and do not have the eigenvalue 1. Further, we suppose that
$C_1+\mu>0$, so that $P_1$ lies on the right side of the border.

Now we need to calculate the minimum number of iteration needed to cross the border again. Let's suppose $n$ iterations are needed to come again to the left side. We
shall sometimes refer to $n$ as the \emph{border return time}.

\begin{equation}
P_n=F^n(P_0)=A^n_r(P_0)+\mu (A_r-I)^{-1}(A^n_r-I)\left(\begin{array}{c}
1 \\
0\\
0 \end{array}\right)
\label{nth_iteration}
\end{equation}

Suppose that $\delta_r\neq 0$. We need to calculate the $n$-th power of a matrix $A$. We here use a recursion method similar to the one used by \cite{avrutin2016dangerous} and calculate $A^n$ for $n\geq 1$ as

\begin{equation}\label{recursive}
A^n=\left(
\begin{array}{ccc}
 a_n & b_n & c_n \\
 d_n & e_n & f_n \\
 g_n & h_n & i_n
\end{array}
\right)=\left(
\begin{array}{ccc}
 a_n & a_{n-1} & a_{n-2} \\
 \delta_r.a_{n-2}-\sigma_r.a_{n-1} & \delta_r.a_{n-3}-\sigma_r.a_{n-2} & \delta_r.a_{n-4}-\sigma_r.a_{n-3} \\
 \delta_r.a_{n-1} & \delta_r.a_{n-2} & \delta_r.a_{n-3}
\end{array}
\right)
\end{equation}

which follows the recursive equation given by
\begin{equation}\label{a_n_recursion}
a_n=\tau_ra_{n-1}-\sigma_ra_{n-2}+\delta_ra_{n-3}
\end{equation}
where the initial conditions are
\[a_0=1, a_{-1}=0, a_{-2}=0, a_{-3}=\frac{1}{\delta_r}\]

Assuming all the iterations are on the right side, after $n$-th iterations the coordinates are given by the following expression:

\begin{eqnarray}
P_n&=&F^n(P_0)=A^n_r(P_0)+\mu (A_r-I)^{-1}(A^n_r-I)\left(\begin{array}{c}
1 \\
0\\
0 \end{array}\right)\nonumber \\
&=&\left(
\begin{array}{ccc}
 a_n & b_n & c_n \\
 d_n & e_n & f_n \\
 g_n & h_n & i_n
\end{array}
\right)\left(\begin{array}{c}
x_0 \\
y_0\\
z_0 \end{array}\right)+ \mu\left(
\begin{array}{ccc}
 a_{11} & a_{12} & a_{13} \\
 a_{21} & a_{22} & a_{23} \\
 a_{31} & a_{32} & a_{33}
\end{array}
\right)\left(
\begin{array}{ccc}
 a_n-1 & b_n & c_n \\
 d_n & e_n-1 & f_n \\
 g_n & h_n & i_n-1
\end{array}
\right)\left(\begin{array}{c}
1 \\
0\\
0 \end{array}\right) \nonumber\\
&=&\left(
\begin{array}{ccc}
 a_n & b_n & c_n \\
 d_n & e_n & f_n \\
 g_n & h_n & i_n
\end{array}
\right)\left(\begin{array}{c}
x_0 \\
y_0\\
z_0 \end{array}\right)+\mu\left(
\begin{array}{ccc}
 a_{11} & a_{12} & a_{13} \\
 a_{21} & a_{22} & a_{23} \\
 a_{31} & a_{32} & a_{33}
\end{array}
\right)\left(\begin{array}{c}
a_n-1 \\
d_n\\
 g_n\end{array}\right)\\
P_n &=&\left(\begin{array}{c}
 \mu(a_{11} (-1+a_n)+a_{12} d_n+a_{13} g_n)+a_n x_0+b_n y_0+c_n z_0 \\
 \mu(a_{21} (-1+a_n)+a_{22} d_n+a_{23} g_n)+d_n x_0+e_n y_0+f_n z_0 \\
 \mu(a_{31} (-1+a_n)+a_{32} d_n+a_{33} g_n)+g_n x_0+h_n y_0+i_n z_0
\end{array}
\right)
\end{eqnarray}

where $(A_r-I)^{-1}=\left(
\begin{array}{ccc}
 a_{11} & a_{12} & a_{13} \\
 a_{21} & a_{22} & a_{23} \\
 a_{31} & a_{32} & a_{33}
\end{array}
\right)$

Note that the recursion relations can be solved explicitly using standard methods (for the 2-dimensional case see \cite{avrutin2016dangerous}, appendix B) and we omit it here.
Also, a recursion scheme to compute the matrix $(A_r-I)^{-1}(A_r^n-I)$ can be given analogous to the one in \cite{saha2015}, appendix E to make the computations more efficient.
With this recursion scheme, we can thus calculate $P_1, P_2, \cdots, P_n$ efficiently until the $x$-coordinate of $P_n$ is negative.
So, we get minimum number $n$ at which the orbit crosses the border again.

%

\textbf{Equation of the stable eigenplane:} $L^*$ is a saddle fixed point which has one unstable eigenvector and two stable eigenvectors.
Equation of the stable eigenplane containing the two stable eigenvectors in vector notation, where $V_2$ and $V_3$ are the two stable eigenvectors is
$$
\overline{r}\cdot (V_2 \times V_3) =0
$$
In Cartesian co-ordinates, we can express this equation as

$$
a.x+b.y+c.z+d=0 \quad \text{ where } (a,b,c)^\intercal:=V_2 \times V_3
$$

As the eigenplane passes through the fixed point $(x_l,y_l,z_l)$, the constant term $d$ can be calculated as $-a.x_l-b.y_l-c.z_l$.

We then check whether the points $P_{n+1}$ and $P_n$ are on the opposite side of the stable eigenplane or not.
If they are on the opposite side, we infer that the unstable manifold must have gone through the stable manifold,
and therefore there is homoclinic intersection resulting the occurrence of chaotic dynamics.

Similar calculation can be carried out for the fixed point at $R^*$. In the next subsection we give a method based on
complex interpolation of the discrete system given by Equation (\ref{normalform}) to estimate the border return time.

\subsection{Complex interpolation and estimation of border return time}

In this subsection we present an interpolation method which considers the positive real exponents of a non-singular real matrix $A$, i.e. matrices of the form $A^t$ for $t\in \mathbb{R}^+$. Note that such exponents are
defined using the complex logarithm of $A$ and the result takes values in the algebra of matrices with complex entries. The matrix $A^t$ is defined by the following formula:

$$
A^t:= \exp(t \ln A)
$$

In general, this is a multi-valued function due to the complex logarithm. However, in this paper we shall always restrict ourselves to the principal value of the complex logarithm
to get a unique matrix $A^t$. Therefore, unless otherwise stated $\ln(z)$ will denote the principal value of the complex logarithm.

For a $3\times 3$-diagonal matrix $D= \begin{bmatrix} z_1 & 0 & 0\\ 0 & z_2 & 0 \\ 0 & 0 & z_3 \end{bmatrix}$ with non-zero complex entries $z_i, i =1, 2, 3$, the matrix $\ln(D)$ is easily defined as
$$
\ln(D): = \begin{bmatrix} \ln(z_1) & 0 & 0\\ 0 & \ln(z_2) & 0 \\ 0 & 0 & \ln(z_3) \end{bmatrix}
$$

Therefore, the matrix $D^t$ for $t\in \mathbb{R^+}$ is given by
$$
D^t:= \begin{bmatrix} \exp(t\ln(z_1)) & 0 & 0\\ 0 & \exp(t\ln(z_2)) & 0 \\ 0 & 0 & \exp(t\ln(z_3)) \end{bmatrix}
$$

\begin{remark}\label{JCdec}
In case of a matrix $A$ having repeated eigenvalues,  given its block-diagonal Jordan normal form, one can also define the matrix $\ln(A)$
using the so-called Jordan-Chevalley decomposition of $A$ into a sum of two matrices $S$ and $N$,
where $S$ is a semisimple matrix and $N$ a nilpotent matrix. For more details, we refer the reader to \cite{HSieh99}, chapter 4. In this paper however, we restrict ourselves to the case of distinct
eigenvalues only.
\end{remark}

Turning to the system given by Equation (\ref{normalform}), we wish to diagonalize the matrices $A_l$ and $A_r$, and then use the above formulae to define the continuous complex interpolation of the system.
Let $P_l$( resp. $P_r$) be the complex base change matrix for $A_l$(resp. $A_r$). Let us assume that neither $A_l$ nor $A_r$ has repeated roots, so that the matrices $D_l= P_l A_l P_l^{-1}$ and
$D_r= P_r A_r P_r^{-1}$ are diagonal (with complex entries). This excludes the case of block-diagonal matrices, although much of the analysis goes through once the matrices $\ln(D_l)$ (or $\ln(D_r)$)
have been defined using the Jordan-Chevalley decomposition, as in Remark (\ref{JCdec}). We also assume as before that the matrices $A_l$ and $A_r$ do not have any eigenvalue equal to 0 or 1. As a consequence,
the matrices $A_l, A_r$, $A_l- I$ and $A_r- I$ are invertible.

For $t \in \mathbb{R}^+$, we set $A_l^t: = P_l^{-1} D^t_l P_l$ and similarly $A_r^t: = P_r^{-1} D^t_r P_r$.
We use the notation $\Re(z)$ to denote the real part of any $m$-tuple of complex numbers $z\in \mathbb{C}^m$.

\begin{definition}
The complex interpolation of the discrete system in Equation (\ref{normalform}) is given by the following formula: for any $t \in [0, 1]$ and
a starting point $X_0 = (x_0, y_0, z_0) \in \mathbb{R}^3$ such that $y_0+\mu \neq 0$,
\begin{equation}\label{complexinterpol}
X(t)\equiv F^t_{\mu}(X_0): =
\begin{cases}
A^t_lX_0+\mu (A_l- I)^{-1}(A_l^t- I)C, &\text{if $x_0 < 0$ or ($x_0=0$ and $y_0 <-\mu$)}\\
A^t_rX_0+\mu (A_r- I)^{-1}(A_r^t- I)C, &\text{if $x_0 > 0$ or ($x_0=0$ and $y_0> -\mu$)} 
\end{cases}
\end{equation}

\end{definition}

The system given by Equation (\ref{complexinterpol}) yields a continuous, piecewise smooth curve $\Re({X(t)})$ (the real part of the complex function $X(t)$) in $\mathbb{R}^3$ for $t \in [0, 1]$. The definition for the complex interpolation for any $t \in \mathbb{R}^+$ is then straightforward to deduce:
$$
F^t_\mu(X_0): = F^{\{t\}}_\mu(F^{\floor{t}}_\mu(X_0))
$$

where $\floor{t}$ is the greatest integer less than or equal to $t$ and $\{t\}$ is the fractional part of $t$. It follows immediately from the definition that the complex interpolation $F^t_\mu$
agrees with the system in Equation (\ref{normalform}) whenever $t$ is an integer. In the definition, we have excluded the line $y_0 =-\mu$ in the $Y$-$Z$ plane so that there is no ambiguity 
in the interpolation; however, the interpolation formulae can be extended to this case according to the situation where the \emph{second} iterate of $X_0$ under $F_\mu$ has $x$-coordinate positive or negative- 
equivalently, whether $z_0>-\mu$ or $z_0 <-\mu$. That leaves us with the case $X_0 = (0, -\mu, -\mu)^\intercal$ for which one should consider on which side of the border the \emph{third} iterate lies. 
This is simply the condition $\mu >0 $ or $\mu <0$. In this way one can extend the interpolation scheme to every point in the $Y$-$Z$ plane.

\begin{definition} We shall call the interpolation $X(t)$ the \emph{companion orbit} of the dynamical system (\ref{normalform}) starting at $X_0$.\\
If $X_0$ lies on the stable (resp. unstable)
 manifold of $L^*$ or $R^*$, we shall call $X(t)$ the \emph{companion stable} (resp. \emph{companion unstable}) manifold.
\end{definition}

Since the companion orbit must agree with the actual orbit at integer points, the following lemma is immediate:

\begin{lemma}
The orbit of a point under $F_\mu$ is bounded if and only if the corresponding companion orbit $F^t_\mu$ is bounded for all $t\in \R^+$.
\end{lemma}

The complex interpolation function $F^t_\mu$ will be used to estimate the border return time in the next subsection. A few remarks are in order.

\begin{remark}
\begin{enumerate}
 \item The terminology `companion stable' does not imply that the curve $X(t)$ itself is stable under iterations of $F_\mu$, rather it is meant to imply that it must intersect with the
 stable manifold infinitely many times. In particular, $F^t_\mu$ does not follow the semi-group composition law $F^{t+s}_\mu=F^t_\mu\circ F^s_\mu$ in general.
 \item We have nevertheless found via numerical simulations that in many cases, the companion orbits (or companion stable or unstable manifolds) do
 capture a lot of global dynamics of the system in Equation (\ref{normalform}), see Figure (\ref{Companionorbit}) below in which a companion orbit
 as well as the companion unstable manifold of $U_l$ for the motivating example in Section 3 is given.
\end{enumerate}

\end{remark}

\begin{figure}
\centering
       \begin{subfigure}[]
       \centering
       \includegraphics[scale=.25]{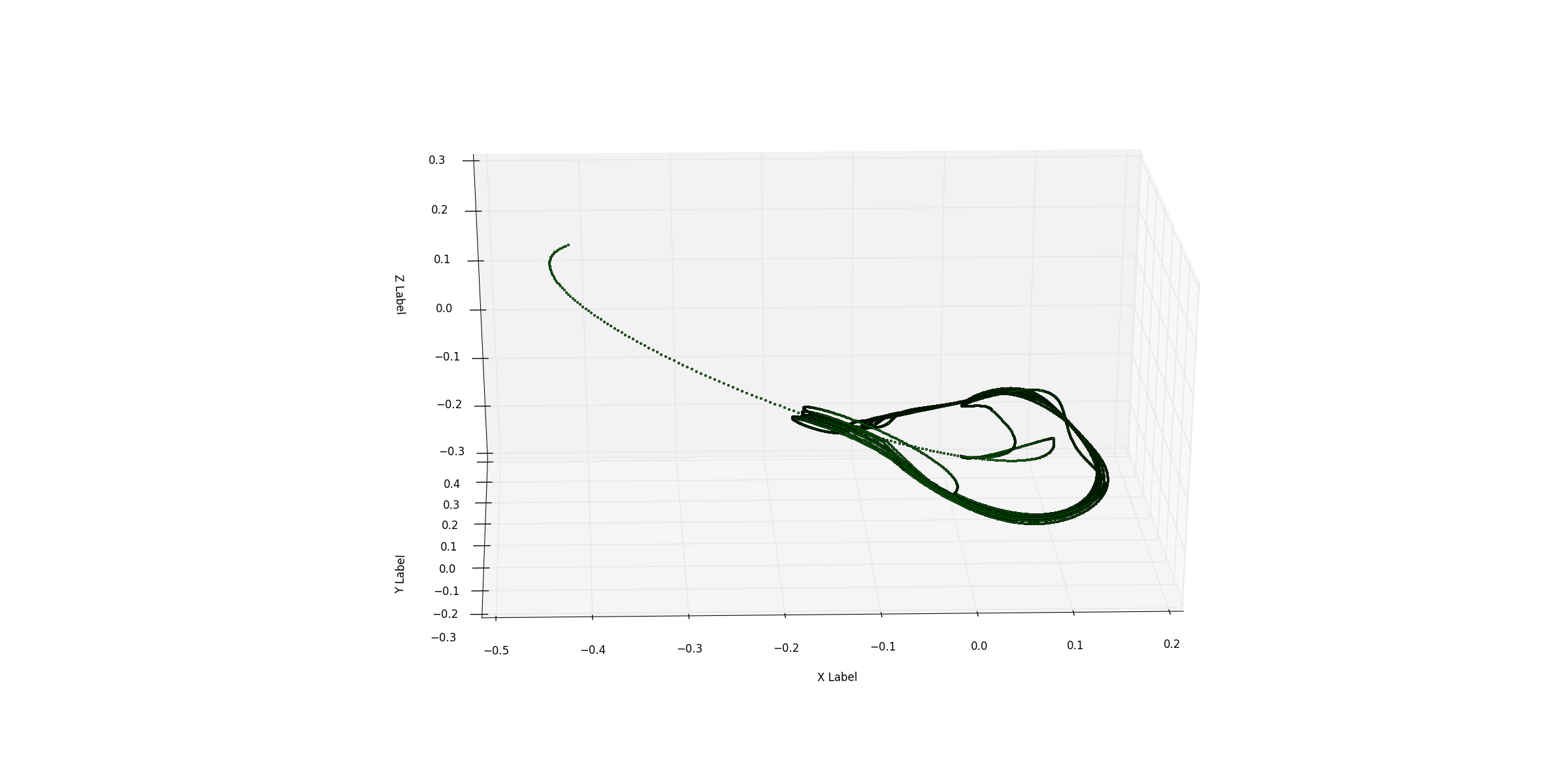}
       \end{subfigure}
       \begin{subfigure}[]
       \centering
       \includegraphics[scale=.25]{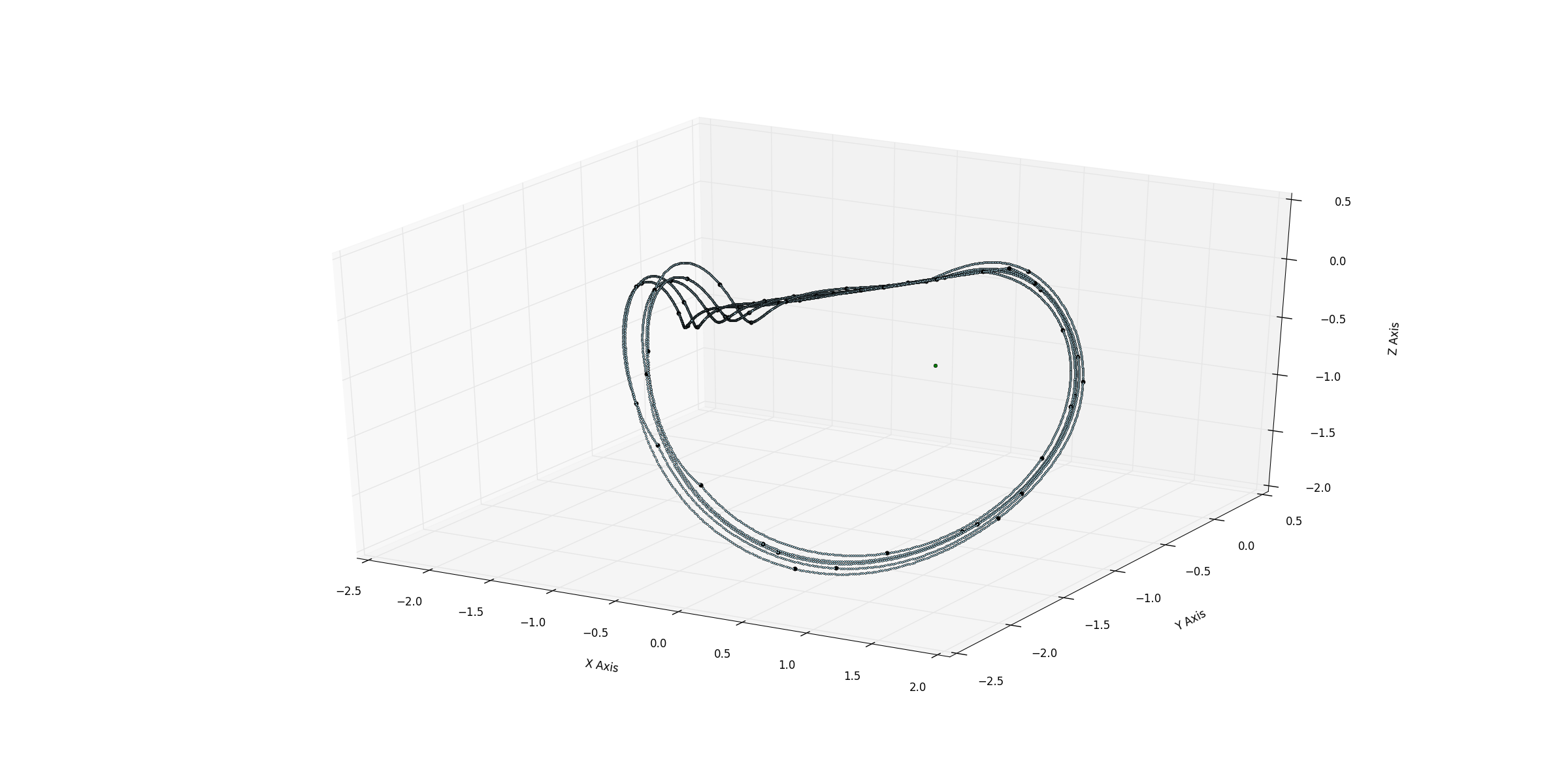}
       \end{subfigure}
\caption{(a) An example of a companion orbit, and (b) The companion unstable manifold for $U_l$ , for the motivating example of Section (\ref{example}). In (b),
the black dots along the companion orbit represent the actual orbit points, while the green dot at the center is the right fixed point $R^*$.}
\label{Companionorbit}
\end{figure}
\subsubsection{Estimating the border return time}
In this section we give some results regarding upper bounds for the border return time as well as algebraic conditions for not returning to the border.
This is important in view of implementing the algorithm based on the recursive Equation (\ref{recursive}) in Subsection (\ref{rec}). Indeed, in order to check the transverse intersection condition by
implementing the algorithm, one must determine in advance an upper bound for the number of iterations $n$ since the border return time equation is transcendental in nature (see Equation
(\ref{borderreturntime})) and therefore has no general solution in closed form. Thus, one has to rely on a computer program to find the least positive solution for such equations.
For such a program, hard-coding a fixed number of iterates to check the border crossing condition for different parameter values is clearly not a reliable method, and this is where upper bounds
for various parameter regions help us determine the number of iterates to check whether a transverse homoclinic intersection exists after the orbit crosses the border.

Let us consider the system in Equation (\ref{normalform}) under the change of co-ordinates corresponding to the eigenbasis of $A_l$. Let us
denote the eigenvalues of $A_l$ by $(\lambda_1, \lambda_2, \lambda_3)$. Let $P_l =(p_{ij})_{i, j =1, 2, 3}$ be the change of basis matrix with respect to the given eigenvalues, and let $P^{-1}_l$
be given by a matrix $(\rho_{ij})_{i, j = 1, 2, 3}$. Thus, the new co-ordinates, which shall be denoted in vector form by $\chi= (\chi_1, \chi_2, \chi_3)^\intercal$ and
the old (standard Euclidean) co-ordinates denoted by $X=(x, y, z)^\intercal$ are related by the change of
base matrix $P_l$:
$$
\chi = P_l^{-1} X \Leftrightarrow \begin{bmatrix} \chi_1 \\ \chi_2 \\ \chi_3 \end{bmatrix}= P_l^{-1}\begin{bmatrix} x \\ y \\ z\end{bmatrix}
$$

Under the $\chi$-coordinates, for a starting point $\xi_0 = (\xi_{01}, \xi_{02}, \xi_{03}) \in \mathbb{R}^3$, the complex interpolation equation takes a particularly simple form. For $i=1,2,3$, we have:

\begin{equation}\label{newinterpol}
\chi_{i}(t) = \lambda_i^t \xi_{0i} + \frac{\lambda^t_i-1}{\lambda_i -1} \mu_i
\end{equation}

where $\mu_i = \mu \rho_{1i}$ for $i =1, 2, 3$.

To compute the border return time for a point $X_0= (0, y_0, z_0)$ in the old-coordinate system for which $y_0 +\mu < 0$ (which means the first iterate lands on the left side of the border), we derive the following equation. Let $\xi_0= P^{-1}_l X_0$. Then, the least
positive real solution of the equation below gives the border return time.

\begin{equation}\label{borderreturntime}
\Re({p_{11} \chi_{1}(t) + p_{12} \chi_2 (t) + p_{13}\chi_3(t)}) =0
\end{equation}

When expressed in the old co-ordinate system, the Equation (\ref{borderreturntime}) takes the following form:

\begin{eqnarray}\label{borderreturntime1}
 &\Re&\Big[ p_{11} ( \lambda_1^t (\rho_{12}y_0+ \rho_{13}z_0) + \frac{\lambda^t_1-1}{\lambda_1 -1} \mu_1) \nonumber \\ \nonumber
&+& p_{12} ( \lambda_2^t (\rho_{22}y_0+ \rho_{23}z_0) + \frac{\lambda^t_2-1}{\lambda_2 -1} \mu_2 ) \\
&+& p_{13}( \lambda_3^t (\rho_{32}y_0+ \rho_{33}z_0) + \frac{\lambda^t_3-1}{\lambda_3 -1} \mu_3 )\Big] =0
\end{eqnarray}

It is straighforward to check that $t=0$ is always a solution to Equation (\ref{borderreturntime1}), 
which corresponds to the starting position $X_0$. As already mentioned, the least positive real solution $t_0$ of Equation (\ref{borderreturntime1})
gives the crossing iteration $n = \floor{t_0}$, i.e. the $n$-th iteration lies on the left side of the border while the $(n+1)$-th iteration passes to the right side of the border.
Note that the above equation interpolates the actual orbit of $X_0$ under $F_\mu$ correctly only for $t\leq t_0$, after which the corresponding equation for the matrix $A_r$ must be used.

If the matrix $A_l$ has all positive real eigenvalues, then of course the expression within the big brackets in Equation (\ref{borderreturntime1}) is already real. However, if any of the eigenvalues
are negative or complex, then there is a non-zero complex part of the interpolation, which we nevertheless ignore.

\begin{remark}\label{flip}
A special case arises when the right fixed point $R^*$ has a flip unstable direction $U_r$ and the point $X_0$ is the intersection point of $U_r$ with the border $x=0$. In that case, the first iterate
will never land on the left side because it flips to the ``opposite side'' of the flip direction with respect to the orientation of $U_r$ on the right side of the border, and the next iterate then lands on $U_r$ again but now on the
left side of the border (since the direction is unstable). Therefore, in this case
the starting point $X_0$ has non-zero $x$-coordinate (in the standard Euclidean basis) and the Equation (\ref{borderreturntime1}) has to be modified accordingly.
This modification is nonetheless straightforward and we omit the details here. The border return time is then estimated as $t_0+2$, where $t_0$ is the least positive solution to the
modified equation. The corresponding function $f(t)$
has the same general form as in the positive eigenvalue case (see Equations (\ref{borderreturntime}) and (\ref{brt2}) below).
\end{remark}

Let us analyse Equation (\ref{borderreturntime1}) in more detail for the special case when $\lambda_1 >1$, and $\lambda_2 = \overline{\lambda}_3 = r_0 e^{i\theta_0}, r_0>0, \theta_0 \in (-\pi, \pi]$ are a pair of complex eigenvalues with absolute
value less than 1. Equation (\ref{borderreturntime1}) then takes the following general form:

\begin{equation}\label{brt2}
f(t):= \alpha_1 \lambda_1^t +2 \alpha_2 \cos(t\theta_0+\delta)  r_0^t +\alpha_3 =0, \quad t \geq 0
\end{equation}

where $\alpha_1, \alpha_2, \alpha_3, \delta \in \mathbb{R}$. Our goal is to give upper bounds for the smallest positive root of $f(t)$.

Note that the cosine term makes $f(t)$
an oscillating function. It has two enveloping curves given by
\begin{eqnarray}\label{bounding}
f^+(t)&=&  \alpha_1 \lambda_1^t +2 \alpha_2  r_0^t +\alpha_3, \quad \text{ and } \nonumber \\
f^-(t)&=&  \alpha_1 \lambda_1^t -2 \alpha_2  r_0^t +\alpha_3
\end{eqnarray}

The condition $f^-(t) \leq f(t) \leq f^+(t)$ is then satisfied for all $t\geq 0$. 
The functions of the kind $f^\pm(t)$ are sometimes called \emph{exponential polynomials } or \emph{generalized Dirichlet polynomials} \cite{jameson2006counting}, \cite{Tossavainen}.
We thus get the following general form of the enveloping curves given in Equation (\ref{bounding}):


\begin{eqnarray}\label{generalexppol}
g(t):= a_1 \kappa_1^t  + a_2 \kappa_2^t + a_3, \quad \quad t \geq 0;  \quad \quad \kappa_1>1> \kappa_2>0,\quad \text{ and } a_i \in \mathbb{R} \text{ for } i=1,2,3.
 \end{eqnarray}

We are interested in the location of zeros of $g(t)$ in the interval $(0, \infty)$. Note that 
for such functions, Descartes' rule of signs is applicable, and we have the following nice theorem, sometimes called the \emph{lost cousin of the Fundamental Theorem of algebra}:

\begin{theorem}\label{fundamental}\cite{jameson2006counting}
Let $\sigma(g)$ be the number of sign changes in the sequence $(a_1, a_2, a_3)$. Then the number of zeros $Z(g)$ of $g(t)$ (counted with multiplicity) in
the interval $(0, \infty)$, is bounded above by $\sigma(g)$. Moreover, the difference $\sigma(g)- Z(g)$ is always an even number, and $Z(g') \geq Z(g)-1$.
\end{theorem}

The following corollaries are immediate.

\begin{corollary}
 With the same assumptions as in Theorem (\ref{fundamental}), the number of zeros of $g(t)$ is at most 2.
\end{corollary}

\begin{corollary}
 With the same assumptions as in Theorem (\ref{fundamental}), if $\sigma(g)=0$ then $g(t)$ has no solution.
\end{corollary}

Note that our original goal was to find roots of the function $f(t)$ of Equation (\ref{brt2}). Since $f^\pm(t)$ bound $f(t)$, the following lemma is obvious:

\begin{lemma} \label{existence1}
If $f(t)$ has at least one solution for $t>0$, then either $f^+(t)$ or $f^-(t)$ must have at least one solution for $t>0$. The least positive solution of $f(t)$ is therefore bounded above
by the largest positive solution of $f^\pm(t)$.
\end{lemma}

Therefore to find an upper bound for the least positive solution of $f(t)$, it suffices to find an upper bound for the largest solution of $f^\pm(t)$, and in turn it suffices to bound the
solutions of the general form $g(t)$. Assuming that at least one solution of $g(t)$ exists in the interval $(0, \infty)$, we give elementary estimates for the upper bound in the following lemma.

\begin{lemma}\label{upper bound}
 If $Z(g)>0$, there is no solution of $g(t)$ in the region $t \geq t_0$, where 
 $$t_0 = \frac{\ln((|a_2|+|a_3|)/|a_1|)}{\ln(\kappa_1)} $$. 
\end{lemma}

\begin{proof}
We treat separately the following two cases:
\begin{enumerate}
 \item $\sigma(g)=1$: we have either (a)  $a_1, a_2>0$, $a_3<0$, or (b) $a_1 >0$, $a_2, a_3<0$, since other cases can be reduced to these two.

 \item $\sigma(g)=2$:  $a_1, a_3>0$, $a_2<0$.

\end{enumerate}

\textbf{Case 1 (a)}: Since we have assumed that at least one solution exists, we must have $|a_3| >a_1$. In this case, we also have $|a_3|>|a_3|-a_2\kappa_2^t$. Therefore there is no solution in the
region $t>t^a_0$, where $t^a_0 = \frac{\ln(|a_3|/a_1)}{\ln(\kappa_1)}$, as in this region $|a_3| \leq a_1 \kappa_1^t$. Since $\kappa_1>1$ and $|a_3|>a_1$, $t^a_0$ is positive.

\textbf{Case 1 (b)}: One can use similar arguments as in the previous case to show that there is no solution for $t>t^b_0 = \frac{\ln((|a_2|+|a_3|)/a_1)}{\ln(\kappa_1)}$.

\textbf{Case (2)}:  Since $a_2<0$ and $\kappa_1 <1$, we have $g'(t)= a_1\ln(\kappa_1)\kappa_1^t + a_2\ln(\kappa_2)\kappa_2^t >0$ for all $t >0$. 
Therefore $g$ can have at most 1 positive solution. Since $\sigma(g)=2$, we must have $Z(g)=0$ by Theorem (\ref{fundamental}), which contradicts our assumption that $Z(g) >0$. 
Therefore the case $\sigma(g)=2$ cannot arise for $g(t)$.

Therefore we get that there is no solution for $g(t)$ for $t > \max\{ t_0^a, t_0^b\} = t^b_0= \frac{\ln((|a_2|+|a_3|)/|a_1|)}{\ln(\kappa_1)}$

%


\end{proof}

Using elementary arguments, one can also give upper bounds for solutions of the border return time Equation (\ref{borderreturntime1}) in many more cases, e.g. when all three
eigenvalues are real. However, we do not claim that any of these upper bounds are tight.

We also note that in case $g(t)$ has exactly one solution, the least positive solution can also be bounded below
by the (unique) extremum point $t^*$ of $g(t)$ given by
$$
t^*= \frac{\ln(-a_2\ln(\kappa_2)/a_1\ln(\kappa_1))}{\ln(\kappa_1/\kappa_2)}
$$

Note that this may or may not be a positive number. 

Now we can summarize the discussions above and state a necessary condition for the occurence of chaos through transverse homoclinic intersection on the 
first return of an unstable manifold $U_r$ for $R^*$, following a Shilnikov-type dynamical behaviour:

\begin{theorem}\label{3dpws_Shilnikov}
Suppose that the left-side fixed point $L^*$ is admissible and has eigenvalues $(\lambda_1, \lambda_2, \lambda_3)$ where $\lambda_1$ is real with $\lambda_1 >1$, and $\lambda_2= \overline{\lambda}_3$
is a pair of complex-conjugate eigenvalues with absolute value $r_0<1$. Assume also that the right-side fixed-point $R^*$ is admissible with one unstable flip eigenvalue and two stable eigenvalues; denote the
associated unstable manifold by $U_r$ and the stable manifold as $S_r$. Take
$X_0 = (0, y_0, z_0)^\intercal$ to be the intersection of $U_r$ with the border $x=0$ with $y_0+\mu <0$, and consider the corresponding function $f(t)$ of Equation (\ref{borderreturntime1})
(see also Remark (\ref{flip})), with upper and lower bounding curves $f^\pm(t)$, with either $f^+(t)$ or $f^-(t)$ having at least one solution (say $f^+(t)$).
Then, a necessary condition for a transverse homoclinic intersection on first return to occur between $U_r$ and $S_r$, is given by:
$$
((P_k-R^*)^\intercal\cdot \overline{n}_{S_r}). ((P_{k+1}-R^*)^\intercal\cdot\overline{n}_{S_r}) <0 \text{ for some } k \leq t_0+2 \text{ where }\\
t_0= \frac{\ln((2|\alpha_2|+|\alpha_3|)/|\alpha_1|)}{\ln(\lambda_1)}.
$$

Here $P_k$ is the $k$-th iterate of $X_0$, $\overline{n}_{S_r}$ is the normal vector to the 2-dimensional stable eigenplane of $R^*$ and  $v^\intercal\cdot w$ denotes the dot product of two vectors $v, w \in \mathbb{R}^3$(
$v^\intercal$ is the transpose of $v$).
\end{theorem}

We would like to apply Theorem (\ref{3dpws_Shilnikov}) to analyse our motivating example of Section 3. However, in that example, $R^*$ has a 1-dimensional \emph{flip stable} manifold $S_r$
and a 2-dimensional \emph{unstable} manifold $U_r$. In that case, our analysis must be modified as follows. Let $X_0 = (0, y_0, z_0)^\intercal$ be the point of intersection of $S_r$
with the border. Since $A_r$ and $A_l$ are invertible (none of its eigenvalues are zero so it is non-singular), the orbit of $X_0$ under the inverse map of $F_\mu$ can be computed easily, which we assume
to lie on the left side of the border. So, to analyse the ``border return time''  corresponding the backward flow of $S_r$ traversing the left side,
we can first modify the interpolation equation for the iterates of the inverse map and 
use a similar formula 
as given in Equation (\ref{complexinterpol}),  then 
analyse the zeros of the function analogous to $f(t)$. One gets in that case another function $G(t)$ which is also of the same
form as $g(t)$ as in Equation (\ref{generalexppol}), 
The rest of the analysis then follows in a similar fashion to find upper bounds of the least positive solution of $G(t)$.

\section{Conclusion}\label{conclusion}
 In the continuous case, Shilnikov bifurcation occurs because the 1-dimensional unstable
 manifold loops back and intersects the 2-dimensional stable manifold. We have shown that a similar phenomenon occurs in a discrete three-dimensional piecewise smooth system, albeit
 the looping of the 1-dimensional unstable manifold occurs due to the nature of the fixed point on the other side of the border.
 Therefore we call this Shilnikov-type dynamics and the resulting chaos as Shilnikov-type chaos.
 In this paper we have also derived analytical conditions for the occurrence of a transverse homoclinic intersection, and therefore the occurrence of a chaotic dynamics.
 It is well-known that such transverse homoclinic intersections cause chaotic behaviour via the presence of Smale
 horseshoes.
 More precisely, we employ two analytical methods to give necessary conditions for the occurrence of a homoclinic orbit: one that uses recursion and another that
 uses complex interpolation. The two methods are complementary to each other and are meant to
 be used together in practice. As far as we know, this method seems to be new in the case of discrete piecewise smooth systems.
 In particular, since the interpolation technique provides an 
 gives an explicitly-defined continuous curve that can be used to probe the dynamics of the system, we expect that it will be useful to study other dynamical properties of such systems as well.
 We also present an example, which shows a `two-sided Shilnikov dynamics', where looping behaviour of 1-dimensional
 eigenmanifolds occurs for both fixed points on either side of the border, which then intersect transversally their respective 2-dimensional eigenmanifolds. The resulting chaotic dynamics
 is Shilnikov-type on either side, which can be therefore aptly named as `two-sided Shilnikov chaos'.

%
%
\bibliographystyle{apalike}

\bibliography{Bibliography}
\end{document}